\documentclass[12pt,letterpaper]{amsart}

\newcounter{myenum}
\newcounter{myenumbak}

\newcommand{\la}{\langle}
\newcommand{\ra}{\rangle}

\newcommand{\sech}{\operatorname{sech}}
\newcommand{\defeq}{\stackrel{\rm{def}}{=}}

\newcommand{\spn}{\operatorname{span}}

\usepackage{amssymb}
\usepackage{amsthm}
\usepackage{amsxtra}
\usepackage{graphicx}

\newtheorem{theorem}{Theorem}

\newtheorem{lemma}[theorem]{Lemma}

\theoremstyle{remark}
\newtheorem{remark}[theorem]{Remark}

\numberwithin{equation}{section}

\numberwithin{theorem}{section}

\numberwithin{table}{section}

\numberwithin{figure}{section}

\ifx\pdfoutput\undefined
  \DeclareGraphicsExtensions{.pstex, .eps}
\else
  \ifx\pdfoutput\relax
    \DeclareGraphicsExtensions{.pstex, .eps}
  \else
    \ifnum\pdfoutput>0
      \DeclareGraphicsExtensions{.pdf}
    \else
      \DeclareGraphicsExtensions{.pstex, .eps}
    \fi
  \fi
\fi

\title
{Soliton dynamics for a non-Hamiltonian perturbation of mKdV}

\author{Quanhui Lin}
\address{Brown University}

\begin{document}

\maketitle

\begin{abstract}
We study the dynamics of soliton solutions to the perturbed mKdV equation $\partial_t u = \partial_x(-\partial_x^2 u -2u^3) + \epsilon V u$, where $V\in \mathcal{C}^1_b(\mathbb{R})$, $0<\epsilon\ll 1$.  This type of perturbation is non-Hamiltonian.  Nevertheless, via symplectic considerations, we show that solutions remain $O(\epsilon \la t\ra^{1/2})$ close to a soliton on an $O(\epsilon^{-1})$ time scale.  Furthermore, we show that the soliton parameters can be chosen to evolve according to specific exact ODEs on the shorter,  but still dynamically relevant, time scale $O(\epsilon^{-1/2})$.  Over this time scale, the perturbation can impart an $O(1)$ influence on the soliton position.
\end{abstract}

\section{Introdution}

We consider the modified Korteweg-de Vries (mKdV) equation with a
small external potential
\begin{equation}
\label{E:pmkdv} \partial_t u = \partial_x(-\partial_x^2 u -2u^3) +
\epsilon V u \,.
\end{equation}
where $0<\epsilon\ll 1$, $V\in \mathcal{C}^1_b(\mathbb{R})$, i.e.
$V$ and $V'$ are continuous and bounded.

The unperturbed case of
\eqref{E:pmkdv},
\begin{equation}
\label{E:mkdv}
\partial_t u = \partial_x(-\partial_x^2 u -2u^3)
\end{equation}
is globally well-posed in $H^k$ for $k\geq 1$ (see Kenig-Ponce-Vega
\cite{KPV}), and possesses single soliton solutions $u(x,t) =
\eta(x,a+c^2t,c)$, for $a\in\mathbb{R}$ and
$c\in\mathbb{R}\,\backslash\,\{0\}$, where $\eta(x,a,c)=c Q(c(x-a))$
with $Q(x)=\sech(x)$ (so that $-Q+Q''+2Q^3 = 0$).   The solitons are
orbitally stable  as solutions to the unperturbed mKdV
\eqref{E:mkdv} (see \cite{Be, Bo, W2, BSS}), i.e. the solutions stay
close to the soliton manifold
$$M = \{ \; \eta(x,a,c) \; | \; a\in \mathbb{R}\,, c>0 \, \}$$
if they are initially close.

Our first main result, Theorem \ref{T:main1}, shows that this type
of orbital stability remains true for the \emph{structurally}
perturbed mKdV  \eqref{E:pmkdv}, in the following sense:  solutions
which start an $H_x^1$ distance $\omega$ from the soliton manifold
$M$ remain within an $H_x^1$ distance $(\omega +\epsilon
t^{1/2})e^{C\epsilon t}$ up to time $\epsilon^{-1}\log
\epsilon^{-1}$.  Our second main result result, Theorem
\ref{T:main}, shows that on the shorter time scale
$\epsilon^{-1/2}\log \epsilon^{-1}$, we can predict the location on
the soliton manifold by solving a system of two ODE for the position
parameter $a$ and scale parameter $c$.  Strong agreement between
this prediction and the numerical solution of \eqref{E:pmkdv} is
illustrated in Fig.~\ref{F:fig1} and Fig.~\ref{F:fig2}. We prove the
global well-posedness of \eqref{E:pmkdv} in $H^1_x$, by adapting the
argument of Kenig-Ponce-Vega \cite{KPV},  in Apx.~\ref{A:wp}.

The forced KdV equation
\begin{equation}
\label{E:fkdv}
\partial_t u = \partial_x(-u_{xx}-3u^2) + \epsilon f
\end{equation}
is a model for free-surface shallow water flow \cite{LYW} with
contributions to $f$ arising from surface pressure and bottom
topography.  Numerics and experiments discussed in \cite{LYW} show
that this type of perturbation can effect the evolution of a single
soliton by generating a procession of small solitons ahead of, and
dispersive waves behind, the primary soliton.

Both \eqref{E:pmkdv} and \eqref{E:fkdv} are specific instances of a
family of gKdV equations with general perturbation
$$\partial_t u = \partial_x(-u_{xx}-u^p) + \epsilon f$$
for $p\in \mathbb{N}$, $p\geq 2$, and $f=f(x,t,u)$.  The case $p=3$
(mKdV) is the unique member of the gKdV family that avoids a certain
anomaly with the symplectic structure.  Specifically, for $p=3$, one
has $\partial_x^{-1}\partial_c \eta \in L^2$ but this fails for
$p\neq 3$.  For $p=3$, one can symplectically project onto the
tangent space of the soliton manifold $M$ rather than on a skew
space.  The difference between $p=3$ and $p\neq 3$ is illustrated in
the fact that the local virial estimate of Martel-Merle \cite{MM1}
simplifies for $p= 3$.  Nevertheless, we believe that the analysis
of the paper carries over in some form to $p\neq 3$ and more general
$f$ of the form $f(x,t,u)$.  We chose \eqref{E:pmkdv} as the
mathematically simplest case in which to illustrate our method.

\subsection{Statements of main results}

\begin{theorem}[orbital stability]
\label{T:main1} Let $\delta>0$ and $a_0,c_0\in\mathbb{R}$ such that
$2\delta \leq  c_0 \leq  (2\delta)^{-1}$.  Suppose $u(x,t)$ solves
\eqref{E:pmkdv} with initial data $u(x,0)$ such that
$$\omega \defeq \|u(x,0) - \eta(x,a_0,c_0)\|_{H_x^1} \lesssim \epsilon^{1/2}$$
Then there exist trajectories $a(t)$ and $c(t)$ so that the
following hold, where $T$ is the maximum time such that $\delta\leq
c(t) \leq \delta^{-1}$ for all $0\leq t \leq T$ and $w(x,t) \defeq
u(x,t) - \eta(x,a(t),c(t))$.  First, we have the following bounds on
the deviation $w$:
\begin{equation}
\label{E:thm1-est-w} \|w\|_{L^\infty_{[0,t]}H^1_x} +
\|e^{-\alpha|x-a|}w\|_{L^2_{[0,t]}H^1_x} \leq C(\omega + \epsilon
t^{1/2}) e^{C\epsilon t}
\end{equation}
Second, we have $T \geq C^{-1}\epsilon^{-1}$ and the
following estimates for the trajectories $a(t)$ and $c(t)$:
\begin{equation}
\label{E:thm1-est-par} \|\dot a - c^2 - \epsilon c^{-1} \la V\eta,
(x-a)\eta\ra\|_{L_{[0,t]}^1\cap L_{[0,t]}^\infty} + \|\dot c -
\epsilon\la V\eta, \eta\ra \|_{L_{[0,t]}^1\cap L_{[0,t]}^\infty}\leq
C (\omega + \epsilon t^{1/2})^2 e^{C \epsilon t}
\end{equation}
The constants $C$ in \eqref{E:thm1-est-w}, \eqref{E:thm1-est-par}
depend on $\|V\|_{\mathcal{C}^1}$ and $\delta$.
\end{theorem}

We remark that the same result holds for $c_0<0$, since $\eta(x,a,-c)=-\eta(x,a,c)$.

\begin{theorem}[exact predictive dynamics]
\label{T:main}
Suppose $u(x,t)$ solves \eqref{E:pmkdv} with initial data $u(x,0)$ satisfying
$$\omega \defeq \|u(x,0) - \eta(x,a_0,c_0)\|_{H_x^1} \lesssim \epsilon^{1/2}$$
where $c_0>0$.  Let $(a(t),c(t))$ evolve according to the ODE system
\begin{equation}
\label{E:thm-dym}
\left\{
\begin{aligned}
& \dot a = c^2 + \epsilon c^{-1}\la V\eta, (x-a)\eta\ra \\
&\dot c = \epsilon\la V\eta, \eta\ra
\end{aligned}
\right.
\end{equation}
with initial data $a(0)=a_0$, $c(0)=c_0$.
Then for $$0\leq t\leq T=\sigma\epsilon^{-1/2}\log\epsilon^{-1}\,,
\qquad\sigma=\sigma(c_0,\|V\|_{\mathcal{C}^1_b})>0\,,$$ we have the
following estimates with $w(x,t)=u(x,t)-\eta(x,a(t),c(t))\,:$
\begin{equation}
\label{E:thm-est-w}
\|w\|_{L^\infty_{[0,t]}H^1_x}+ \|e^{-\alpha|x-a|}w\|_{L^2_{[0,t]}H^1_x}\leq
C( \omega + \epsilon  t^{1/2}) e^{C\epsilon^{1/2}t}\,.
\end{equation}
where $C=C(c_0,\|V\|_{\mathcal{C}^1_b})$.
\end{theorem}
We remark that if one selects initial data so that  $\omega \lesssim
\epsilon^{3/4}$, then the two terms on the right-side of the
estimate \eqref{E:thm-est-w} balance on the $\epsilon^{-1/2}$ time
scale.  In this case the bound becomes $\epsilon^{3/4}
e^{C\epsilon^{1/2}t}$.

\begin{figure}
 \includegraphics[scale=.65]{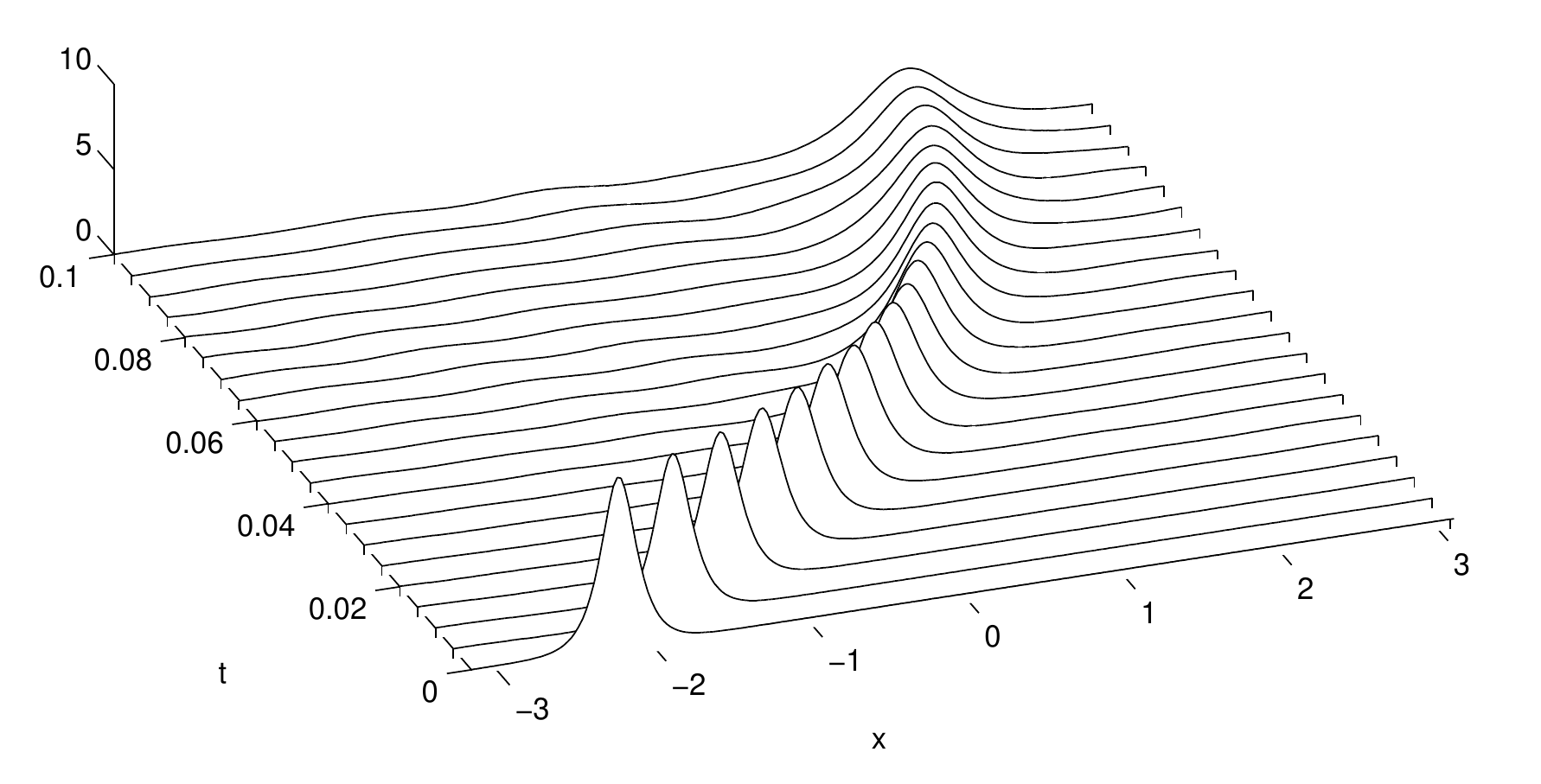}
 \includegraphics[scale=.8]{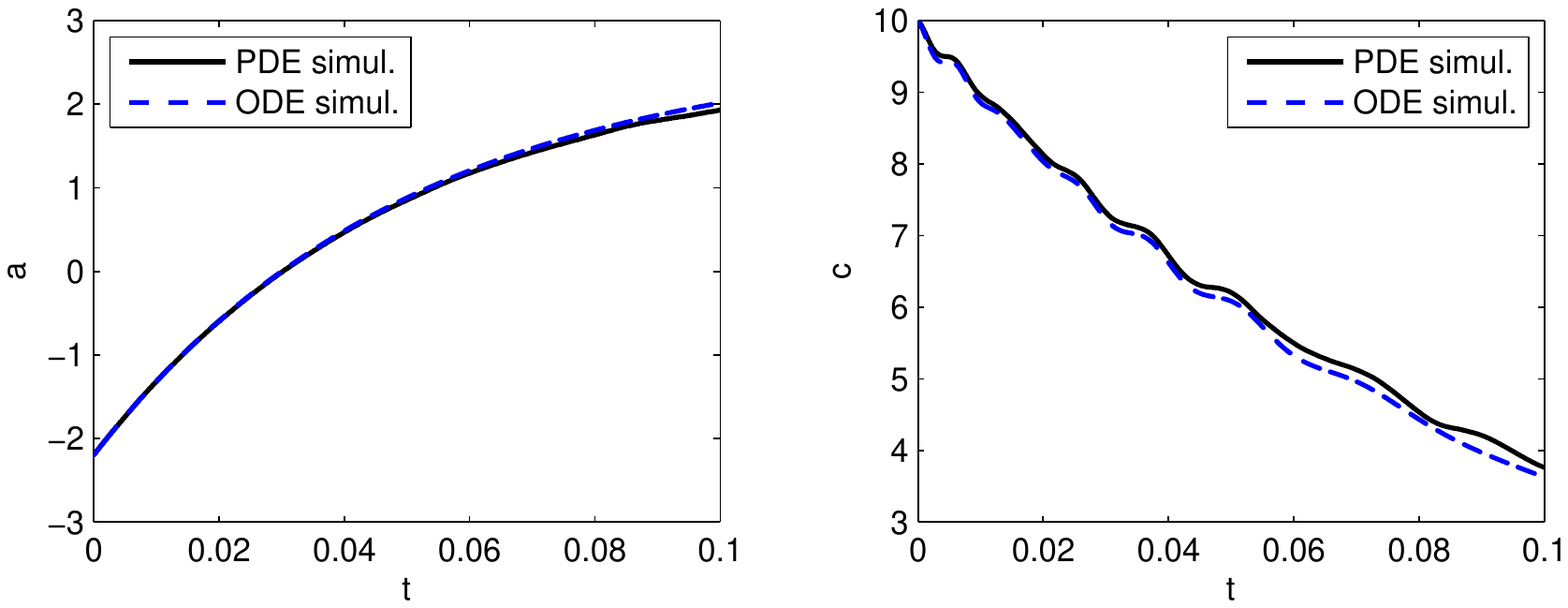}
 \caption{With external potential given by $V_1$ as in \eqref{E:potential1}, the top plot gives the
 rescaled evolution $U(X,T)$, the bottom two plots give the comparison between the evolution of the
 parameters obtained solving the ODE system and exact PDE evolution, i.e. we fit the solution to
 $\eta(X,\tilde A,\tilde C)$, and plot $T$ versus $\tilde A$ and $\tilde C$ respectively.}
 \label{F:fig1}
\end{figure}

\subsection{Relation to recent work}

The energy-Lyapunov based methods for proving orbital stability of
solitons subject to perturbations (of the data, as opposed to the
structural perturbations considered here) were developed by Benjamin
\cite{Be}, Bona \cite{Bo}, Weinstein \cite{W2},
Grillakis-Shatah-Strauss \cite{GSS1, GSS2}.  In the last decade
several results have emerged using the same basic framework to
address the dynamics of solitons for equations subject to structural
perturbations \cite{BJ, FGJS, FTY, H, HL, HPZ, HZ1, HZ2, DV, AS-S,
AS, M1, M2}. The nonlinear Schr\"odinger equation (NLS) with slowly
varying potential was considered by
Fr\"ohlich-Gustafson-Jonsson-Sigal \cite{FGJS} and a result of
``orbital stability'' type was obtained, however the estimates were
not strong enough to obtain ``exact predictive dynamics''.
Holmer-Zworski \cite{HZ2} obtained exact predictive dynamics plus
refined accuracy by adopting the conceptual perspective of
symplectic projection, but also, at the technical level, finding an
appropriate distortion of the soliton manifold that enabled refined
Lyapunov estimates.  This ``symplectic projection plus correction
term method'' has been subsequently pursued in different contexts in
Datchev-Ventura \cite{DV}, Holmer-Lin \cite{HL},
Holmer-Perelman-Zworski \cite{HPZ}, and Pocovnicu \cite{P}.  To
treat a problem in which the perturbation gives rise to significant
dispersive radiation, a different approach was employed by Holmer
\cite{H}.  He treated the KdV equation with a slowly varying
potential, and used the Martel-Merle local virial estimate
\cite{M1,M2} to supplement the energy Lyapunov estimate.  In this
paper, we follow this approach as well.  We show the method is
sufficiently robust to handle small non-Hamiltonian perturbations,
which had not been considered in any of the above papers.  A
stochastic variant of the problem we consider has been addressed by
de Bouard--Debussche \cite{deB-Deb} without the use of the local
virial estimate.  Work in progress by Holmer-Setayeshgar \cite{HS}
will adapt the present paper to the stochastic setting and obtain a
refinement of the results of \cite{deB-Deb}.

\subsection{Numerics}
To solve \eqref{E:pmkdv} numerically we adapt the method in \cite{T}
which is based on the fast fourier transform in $x$, then
fourth-order Runge-Kutta for the resulting ODE in $t$. We use the
rescaled coordinate frame $X=\epsilon^{-1/3}x$, $T=\epsilon^{-1} t$,
and consider the equation on $[-\pi,\pi)$. If $U(X,T)$ solves
$$\partial_TU = \partial_X(-\partial_X^2U-2U^3)+V(X)U\,,$$
with initial data
$$U(0,X) = \eta(X,A_0,C_0)
 = \eta(X,\epsilon^{1/3}a_0,\epsilon^{-1/3}c_0)\,,$$
then $u(x,t)=\epsilon^{1/3}U(\epsilon^{1/3}x,\epsilon t)$ gives a
solution of \eqref{E:pmkdv} on
$[-\pi/\epsilon^{1/3},\pi/\epsilon^{1/3})$ with initial data
$u(0,x)=\eta(x,a_0,c_0)$, and periodic boundary conditions.
Fig.~\ref{F:fig1} and Fig~\ref{F:fig2} plot the evolution of the
soliton initial data (after rescaling) in the following external
potential respectively
\begin{equation}
\label{E:potential1} V_1=-10\cos^2(6X) + 6\sin(10X)\,,
\end{equation}
\begin{equation}
\label{E:potential2} V_2=8\cos^2(4X) - 4\sin(2X)\,.
\end{equation}
Note that to examine the solution $u(x,t)$ on time interval $0\leq
t\leq C\epsilon^{-1/2}$(or $C\epsilon^{-1}$), we should let $U(X,T)$
evolve for time $C\epsilon^{1/2}$(or $C$).

\begin{figure}
 \includegraphics[scale=.65]{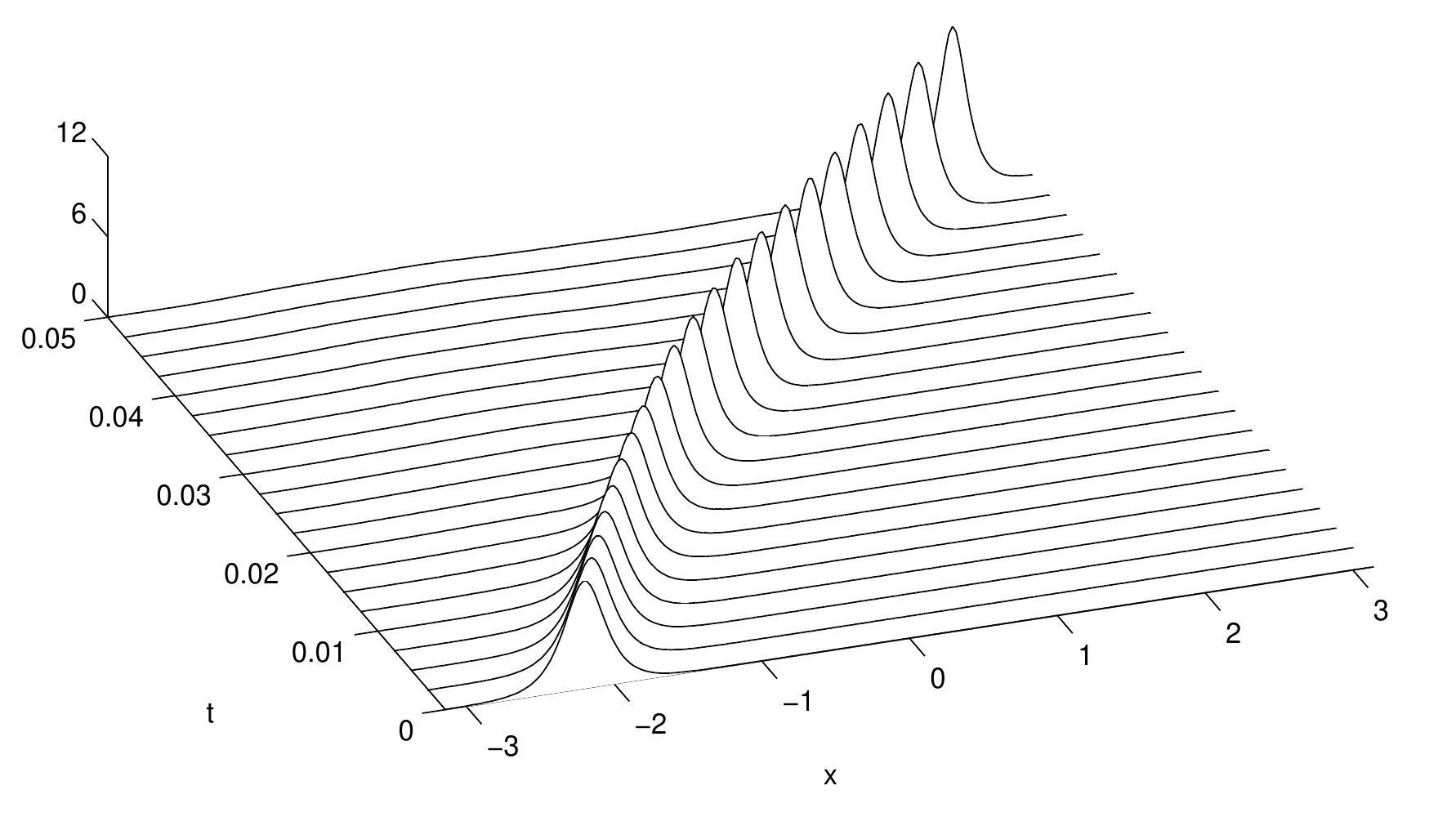}
 \includegraphics[scale=.8]{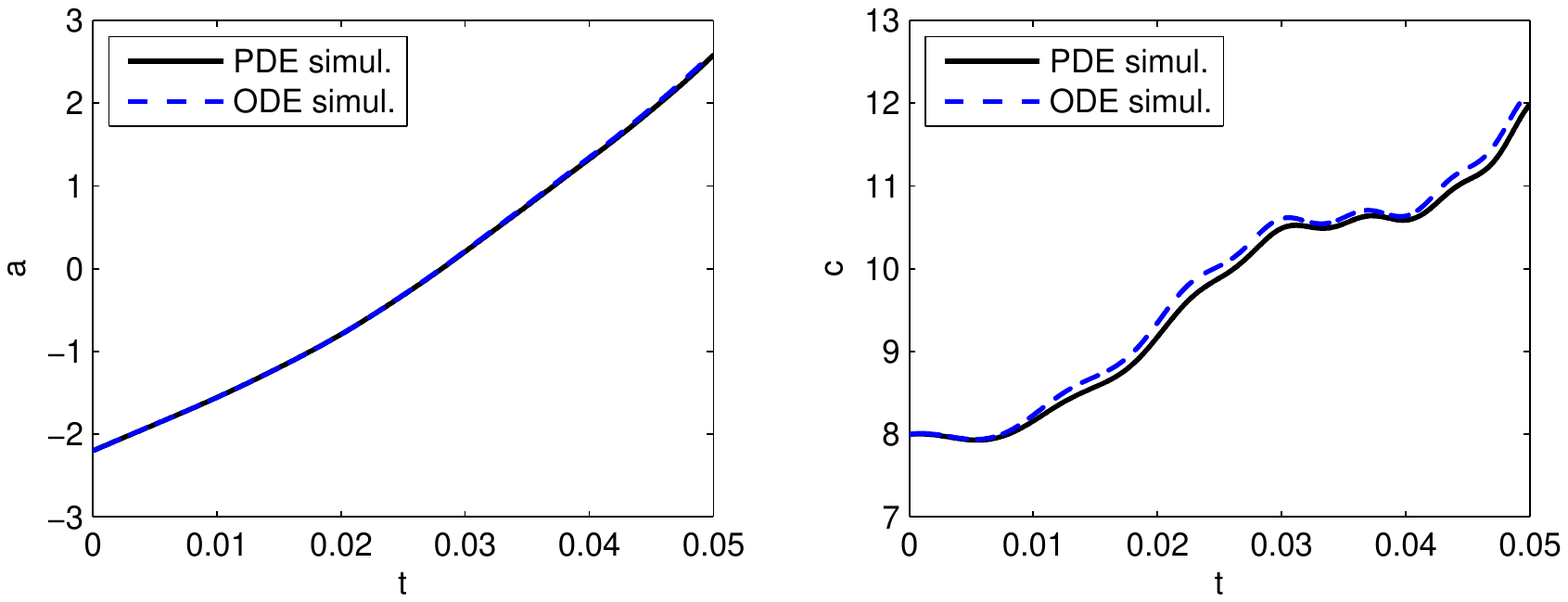}
 \caption{These plots are analogs of Fig.~\ref{F:fig1}, the external potential is given by
 $V_2$ as in \eqref{E:potential2}.}
 \label{F:fig2}
\end{figure}

\subsection{Acknowledgements}
The author thanks his advisor Prof. Justin Holmer for all the
helpful discussions and encouragements.

\section{Background on Hamiltonian structure}
\label{S:background}

Let $J=\partial_x$, and consider
$L^2(\mathbb{R}\mapsto \mathbb{R})$ as a manifold with metric $\la
v_1,v_2\ra = \int v_1  v_2\,dx$, we can define the symplectic form
as
\begin{equation}
\label{E:symp} \omega(v_1,v_2) = \la v_1 J^{-1} v_2 \ra = \la v_1,
\partial_x^{-1}v_2\ra \,,
\end{equation}
where $J^{-1}$ is given by
$$J^{-1}f(x)=\partial_x^{-1}f(x)\defeq\frac12\left(\int_{-\infty}^x
- \int_x^{+\infty}\right)f(y)\,dy\,.$$ The mKdV equation
\eqref{E:mkdv} is the Hamiltonian flow associated with
$$H_0(u) = \frac12\int (u_x^2-u^4)\,,$$
i.e. we can write \eqref{E:mkdv} as
\begin{equation}
\label{E:mkdv2}\partial_t u = JH_0'(u)\,.
\end{equation}
Solutions to mKdV also satisfy conservation of mass $M(u)$ and
momentum $P(u)$, where
$$M(u) = \int u\, dx \,, \qquad P(u) = \frac12\int u^2\,dx \,.$$
We define 2-dimensional manifold of solitons $M$ as
$$M = \{ \, \eta(\cdot, a, c) \, | \,\,a\in \mathbb{R}, c\in \mathbb{R}\,\backslash\, \{0\}\} \,.$$
The symplectic form \eqref{E:symp} restricted to $M$ is given
by $\omega|_M=da\wedge dc$. We denote $\eta=\eta(\cdot,a,c)$, the
dependence of $(a,c)$ on $\eta$ is always meant implicitly. The
tangent space at $\eta$ is given by
$$T_\eta M = \spn \{ \, \partial_a \eta, \partial_c \eta \, \}\,.$$
Note that $JH_0'(\eta) \in T_{\eta} M$, thus the flow associated to
\eqref{E:mkdv} will remain on $M$ if it is initially. Specifically,
direct computation shows
\begin{equation}
\label{E:JH0-prime} JH_0'(\eta) = c^2\partial_a\eta \,.
\end{equation}
which, together with \eqref{E:mkdv2}, explains the form of the
expression for single solitons.  This is equivalent to saying that
the flow \eqref{E:mkdv2} restricted to $M$ (and thus stays on $M$)
is given by
\begin{equation}
\label{E:free-flow} \left\{
\begin{aligned}
&\dot a = c^2 \\
&\dot c = 0
\end{aligned}
\right.
\end{equation}
One can also get \eqref{E:free-flow} by  first restricting $H_0$ to
$M$ to obtain
$$H_0(\eta) = -\frac13 c^3 \,,$$
and then noticing that \eqref{E:free-flow} is just the solution to
the Hamilton equations of motion for $H_0(\eta)$ with respect to
$\omega|_M$:
\begin{equation}
\label{E:free-ODEs} \left\{
\begin{aligned}
&\dot a = -\frac{\partial H_0}{\partial c} = c^2 \\
&\dot c = \frac{\partial H_0}{\partial a} = 0
\end{aligned}
\right.
\end{equation}
Note that we can write \eqref{E:JH0-prime} as
\begin{equation}
\label{E:JH0-prime-2} JH_0'(\eta) + c^2 JP'(\eta)=0 \,.
\end{equation}
From this, we learned that $L'(\eta) =0$, where
\begin{equation}
\label{E:L-classical} L(u) \defeq H_0(u) + c^2 P(u) \,.
\end{equation}
which is the Lyapunov functional used in the classical orbital
stability theory, see \cite{W2}.

Next, we define the symplectic orthogonal projection operator at
$(a,c)$:
$$\Pi_{a,c}: L^2 \cong T_{\eta}L^2 \to T_\eta M\,,$$
by requiring that
$$\la \Pi_{a,c}^\perp f, J^{-1}\partial_a \eta\ra = \la
\Pi_{a,c}^\perp f, J^{-1}\partial_c \eta\ra =0\,,$$ where
$\Pi_{a,c}^\perp  = I - \Pi_{a,c}$, equivalently,
$$\Pi_{a,c}f = \la
 f, J^{-1}\partial_c \eta\ra\partial_a\eta - \la f, J^{-1}\partial_a \eta\ra\partial_c\eta\,.$$
Note that for mKdV,
$$J^{-1}\partial_a\eta = -\eta\text{   and   }J^{-1}\partial_c\eta =
c^{-1} (x-a)\eta\,.$$

\section{Decomposition of the flow}
\label{S:eff-dyn} We can arrange the modulation parameters $a(t)$
and $c(t)$ so that
$$\Pi_{a(t),c(t)}\left[u(x,t) - \eta(x,a(t),c(t))\right]=0\,.$$
This is a standard fact and we recall it in the following
\begin{lemma}
\label{L:orth-decomp} Given $\tilde a$, $\tilde c$, there exist
$\delta_1>0$, $C>0$, such that if $u = \eta(\cdot, \tilde a,\tilde
c)+\tilde w$ with $\|\tilde w\|_{H^1_x}\leq\delta_1$, then there
exist unique $a$, $c$ such that
\begin{equation}
\label{E:decomp}w(x,t) \defeq u(x,t) - \eta(x,a(t),c(t))
\end{equation}
satisfies the symplectic orthogonality conditions
\begin{equation}
\label{E:orth} \la w,J^{-1}\partial_a\eta\ra = \la w,
J^{-1}\partial_c\eta\ra = 0\,.
\end{equation}
Moreover, $$|a-\tilde a|\leq C\|\tilde w\|_{H^1_x}\,,\ \  |c-\tilde
c|\leq C\|\tilde w\|_{H^1_x}\,.$$
\end{lemma}
\begin{proof}
Define
$\phi:\,H^1_x\times\mathbb{R}\times\mathbb{R}\rightarrow\mathbb{R}^2$
by
$$\phi(v,a,c) = \begin{bmatrix}
\la v-\eta,\eta\ra \\
\la v-\eta,(x-a)\eta\ra
\end{bmatrix}
$$
Using $\omega|_M=da\wedge dc$, we can get the Jacobian matrix of
$\phi$ with respect to $(a,c)$ at $(\eta(\cdot,\tilde a,\tilde
c),\tilde a,\tilde c)$
$$(D_{a,c}\phi)(\eta(\cdot,\tilde a,\tilde
c),\tilde a,\tilde c) =
\begin{bmatrix}
0 & 1 \\
1 & 0
\end{bmatrix}
$$
which implies, by the implicit function theorem, that the equation
$\phi(u,a,c)=0$ can be solved for $(a,c)$ in terms of $u$ in a
neighborhood of $\eta(\cdot, \tilde a,\tilde c)$.
\end{proof}
Now since $u=w+\eta$ and $u$ solves \eqref{E:pmkdv}, we compute
\begin{equation}
\begin{aligned}
\label{E:lin-w}
\partial_t w& = \partial_x(-\partial_x^2w-6\eta^2w-6\eta w^2-2w^3)+\epsilon Vw
-F_0\\&=\partial_x(\mathcal{L}w-c^2 w - 6\eta w^2-2w^3)+\epsilon Vw
-F_0\,,
\end{aligned}
\end{equation}
where
$$\mathcal{L} = -\partial_x^2-6\eta^2+c^2\,,$$
and $F_0$ results from the perturbation and $\partial_t$ landing on
the parameters:
$$F_0 = (\dot a-c^2)\partial_a\eta + \dot c\partial_c \eta -
\epsilon V \eta\,.$$
Next, decompose $F_0$ into the symplectically parallel part
$\Pi_{a,c} F_0$ and symplectically orthogonal part $\Pi_{a,c}^\perp
F_0$, explicitly,
\begin{equation}
\label{E:F0-paralel} \Pi_{a,c} F_0 = (\dot a - c^2 - \epsilon \la
V\eta, J^{-1}\partial_c\eta\ra)\partial_a\eta + (\dot c + \epsilon
\la V\eta,J^{-1}\partial_a\eta\ra)\partial_c\eta\,,
\end{equation}
\begin{equation}
\label{E:F0-perp} \Pi_{a,c}^\perp F_0 = -\epsilon V\eta + \epsilon
\la V\eta, J^{-1}\partial_c\eta\ra\partial_a\eta - \epsilon \la
V\eta,J^{-1}\partial_a\eta\ra\partial_c\eta\,.
\end{equation}
We now obtain the equations for the parameters:
\begin{lemma}[effective dynamics]
\label{L:ODEs} Given $V\in\mathcal{C}^1_b$, suppose that $w$ defined
by \eqref{E:decomp} satisfies the orthogonality conditions
\eqref{E:orth}.  Then there exists $\alpha>0$ such that
\begin{equation}
\label{E:eff-dyn} \| \partial_t \eta - c^2\partial_a\eta -
\epsilon\Pi_{a,c}(V\eta)  \|_{T_{a,c}M} \lesssim
\|e^{-\alpha|x-a|}w\|_{H^1}^2 +
\epsilon\|e^{-\alpha|x-a|}w\|_{H^1}\,.
\end{equation}
Explicitly,
\begin{equation}
\label{E:eff-dyn2} \begin{aligned}&\left|\dot a - c^2 - \epsilon\la
V\eta, J^{-1}\partial_c\eta\ra\right| \lesssim
\|e^{-\alpha|x-a|}w\|_{H^1}^2
+ \epsilon\|e^{-\alpha|x-a|}w\|_{H^1}\,,\\
&\left|\dot c + \epsilon\la V\eta, J^{-1}\partial_a\eta\ra\right|
\lesssim \|e^{-\alpha|x-a|}w\|_{H^1}^2 +
\epsilon\|e^{-\alpha|x-a|}w\|_{H^1}\,.
\end{aligned}
\end{equation}
\end{lemma}
As all norms on a finite dimensional space are equivalent, we can take
$$\| \alpha \partial_a \eta + \beta \partial_c \eta \|_{T_{a,c}M} = |\alpha| + |\beta|$$
\begin{proof}
Recall that $$\partial_t w = JH_0''(\eta)w - J(6\eta w^2 - 2w^3)
+\epsilon V w - F_0\,.$$ Write $\mathcal{R}$ for the error terms of
the same order as the right hand side of \eqref{E:eff-dyn}, take
derivative with respect to $t$ for $\la w, J^{-1}\partial_a\eta\ra$,
we have
\begin{equation}
\label{E:der-a}
\begin{aligned}
\  &0=\la\partial_tw,J^{-1}\partial_a\eta\ra + \la
w,J^{-1}\partial_a\partial_t\eta\ra\\ =& -\la
F_0,J^{-1}\partial_a\eta\ra + \la
JH_0''(\eta)w,J^{-1}\partial_a\eta\ra + \la
w,J^{-1}\partial_a\partial_t\eta\ra + \mathcal{R} \\=& -\la
F_0,J^{-1}\partial_a\eta\ra + \la w,
J^{-1}\partial_a(\partial_t\eta-JH_0'(\eta))\ra+\mathcal{R} \\=
&-\la F_0,J^{-1}\partial_a\eta\ra + \la w,
J^{-1}\partial_a(\Pi_{a,c}F_0)\ra + \mathcal{R}\,,
\end{aligned}
\end{equation}
where for the penultimate equality we have used $J^*J^{-1}=-I$ and
the self-adjointness of $H_0''$, and for the last that
$$\partial_t\eta-JH_0'(\eta) = (\dot a - c^2)\partial_a\eta + \dot c\,\partial_c\eta = \Pi_{a,c}F_0
+ O(\epsilon)\partial_a\eta + O(\epsilon)\partial_c\eta\,.$$

Taking derivative for $\la w, J^{-1}\partial_c\eta\ra$, similar
computation gives
$$0=-\la F_0,J^{-1}\partial_c\eta\ra + \la w,
J^{-1}\partial_c(\Pi_{a,c}F_0)\ra + \mathcal{R}\,.$$ Combining with
\eqref{E:der-a}, and applying the orthogonality conditions for the
second terms when $\partial_a$ and $\partial_c$ land on the
coefficients of $\Pi_{a,c}F_0$, the lemma follows from
Cauchy-Schwarz and the smallness of $w$.
\end{proof}

\section{Local virial estimate}
\label{S:virial}

In this section we review, and then apply, part of the local virial
estimates due to Martel-Merle. Let $\Phi\in\mathcal{C}(\mathbb{R})$,
$\Phi(x)=\Phi(-x)$, $\Phi^\prime\leq0$ on $(0,\infty)$, such that
$$\Phi(x)=1 \text{ on }[0,1]\,,\qquad\Phi(x)=e^{-x} \text{ on }[2,\infty)\,,\qquad e^{-x}
\leq\Phi(x)\leq 3e^{-x}\text{ on }[0,\infty)\,.$$ Let
$\Psi(x)=\int_0^x\Phi(y)\,dy$, and for $A\gg0$, set
$\Psi_A(x)=A\Psi(x/A)$, we have following
\begin{lemma} [Martel-Merle \cite{MM1,MM2} local virial spectral estimate]
\label{L:virial}There exists $A$ sufficiently large and $\lambda_0$
sufficiently small, such that if $w$ satisfies the orthogonal
condition \eqref{E:orth}, then
\begin{equation*}
-\la\Psi_A(x-a)w,\partial_x\mathcal{L}w\ra\geq
\lambda_0\int(w_x^2+w^2)e^{-|x-a|/A}\,dx\,.
\end{equation*}
\end{lemma}
Denoting $\psi(\cdot)$ for $\Psi_A(\cdot-a)$, we now proceed as in
\cite{MM1}:
\begin{lemma}[local virial estimate]
\label{L:app-virial} Suppose $V$ is bounded, then there exist
$\alpha>0$ and $\kappa_j>0$, $j=1,2$, such that if $w$ solves
\eqref{E:lin-w} and satisfies the orthogonality conditions
\eqref{E:orth}, then
\begin{equation}
\label{E:app-virial-diff} \|e^{-\alpha|x-a|}w\|^2_{H^1_x}\leq
-\kappa_1\partial_t\int\psi w^2\,dx
+\kappa_2\epsilon^2+\kappa_2\epsilon\|w\|^2_{H^1_x}\,.
\end{equation}
\end{lemma}
\begin{proof}
From the equation for $\partial_t w$, we have
\begin{align*}
\partial_t\int\psi w^2 =
& -\dot a\int\psi^\prime w^2 + 2\int\psi w\partial_t w\\
=& -\dot a\int\psi^\prime w^2 + 2\int\psi w\partial_x(\mathcal{L}w)
&&\leftarrow\text{I}+\text{II}\\
&- 2c^2\int\psi w\partial_x w - 12\int\psi w\partial_x(\eta
w^2)&&\leftarrow\text{III}+\text{IV}
\\ &- 4\int\psi w\partial_x(w^3) +
2\epsilon\int\psi Vw^2 - 2\int\psi wF_0
&&\leftarrow\text{V}+\text{VI}+\text{VII}
\end{align*}
Using integration by parts,
$$\text{III} = c^2\int\psi^\prime w^2\,,$$
hence
\begin{equation}
\label{E:v10} |\text{I}+\text{III}|=|-(\dot a-c^2)\int\psi^\prime
w^2|\lesssim
\epsilon\|w\|^2_{H^1_x}+\|e^{-\alpha|x-a|}w\|_{H^1_x}^2\|w\|^2_{H^1_x}
\end{equation}
by \eqref{E:eff-dyn2}. Following from the boundedness of $\psi$ and
$V$, and the estimate $\|w\|_{L^\infty_x}\lesssim\|w\|_{H^1_x}$, we
obtain
\begin{equation}
\label{E:v20}
\begin{aligned}
|\text{IV}|&\lesssim\|e^{-\alpha|x-a|}w\|_{H^1_x}^2\|w\|_{H^1_x}\,,\\
|\text{V}|=\big|3\int\psi' w^4\big|&\lesssim \|w\|_{H^1_x}^2\|e^{-|x-a|/(2A)}w\|_{L^2_x}^2\,,\\
|\text{VI}|&\lesssim \epsilon\|w\|_{H^1_x}^2\,,
\end{aligned}
\end{equation}
where for the second estimate we have used $\psi'=\Phi((x-a)/A)$ and
the definition of $\Phi$. Decomposing VII term as
$$\text{VII} =-2\int\psi w\Pi F_0 - 2\int\psi w\Pi^\perp F_0  = \text{VIIA}+\text{VIIB}\,,$$
we have by Lemma \ref{L:ODEs} that
\begin{equation}
\label{E:v30} \text{VIIA}\lesssim \epsilon\|w\|^2_{H^1_x} +
\|e^{-\alpha|x-a|}w\|_{H^1_x}^2\|w\|_{H^1_x}\,,
\end{equation}
and by $\Pi^\perp F_0\sim \epsilon\eta$ (see \eqref{E:F0-perp}) that for any $\mu>0$,
\begin{equation}
\label{E:v40} \text{VIIB} \lesssim
\epsilon\|e^{-\alpha|x-a|}w\|_{H^1_x} \lesssim \mu^{-1}\epsilon^2 + \mu \| e^{-\alpha|x-a|}w\|_{H_x^1}^2
\end{equation}
Note in above estimates the value of $\alpha$ may change from one
line to the next, but we can choose one single small enough $\alpha$
that works for all.

By Lemma \ref{L:virial}, we have
$$\text{II} = 2\la\psi w,\partial_x(\mathcal{L}w)\ra\leq
-\lambda_0\int(w_x^2+w^2)e^{-|x-a|/A}\,dx\,.$$  Combining with
\eqref{E:v10}, \eqref{E:v20}, \eqref{E:v30} and \eqref{E:v40}, the
estimate \eqref{E:app-virial-diff} follows by the smallness of
$\|w\|_{H^1_x}$, taking $A$ large enough
so that $1/(2A)<\alpha$, and $\mu>0$ suitably small.
\end{proof}

\section{Energy estimate}
\label{S:energy}

In this section we formulate the energy estimate
necessary for the estimation of the error term $w$. Recall
$\mathcal{L} = -\partial_x^2-6\eta^2+c^2$.  Let
$$\mathcal{E} = \frac12\la\mathcal{L}w,w\ra-2\int\eta w^3\,dx-\frac12\int w^4\,dx\,,$$

Note that $\mathcal{L}=H''_0(\eta)+c^2=L''(\eta)$, see \eqref{E:JH0-prime-2}
and \eqref{E:L-classical}.  We have classical coercivity properties
for $\mathcal{L}$ (for a proof, see e.g. \cite[Prop $2.9$]{W1} or
\cite[Prop $4.1$]{HZ1} for a more direct proof -- note that $\mathcal{L}$ is
the operator $L_+$ considered there):
\begin{lemma}[energy spectral estimate]
\label{L:coer}
Suppose that $w$ satisfies the orthogonality condition \eqref{E:orth}.  Then
\begin{equation}
\label{E:coer} \la \mathcal{L}w,w\ra \gtrsim \|w_x\|_{L^2}^2 + c^2\|w\|_{L_x^2}^2\,,
\end{equation}
\end{lemma}

Since we impose a lower bound on $c$ in Theorems \ref{T:main1}, it
follows from \eqref{E:coer} that if $\|w\|_{H_x^1}$ is smaller than
some ($\epsilon$ independent) constant, then
$$\|w(t) \|_{H_x^1}^2 \sim \mathcal{E}(t)$$

\begin{lemma}[energy estimate]
\label{L:energy} Suppose we are given $V\in\mathcal{C}_b^1$,
$\delta_0>0$ and $w(x,t)$, such that $\delta_0<c(t)<\delta_0^{-1}$,
$w$ solves \eqref{E:lin-w} and satisfies the orthogonality
conditions \eqref{E:orth}, then
\begin{equation}
\label{E:energy}
|\partial_t\mathcal{E}|\lesssim\epsilon\|w\|^2_{H^1_x} +
\epsilon\|e^{-\alpha|x-a|} w\|_{H^1_x} + \|w\|_{H^1_x}^2
\|e^{-\alpha|x-a|}w\|_{H^1_x}^2 + \|w\|_{H^1_x}^6\,.
\end{equation}
where the implicit constant depends on $\delta_0$, $\sigma_0$ and
the bounds on $V$ and $V'$.
\end{lemma}
\begin{proof}
We compute
\begin{align*}
\partial_t\mathcal{E} = &\la\mathcal{L} w,\partial_t w\ra + \dot c
c\|w\|_{L^2_x}^2-6\la(\dot a\partial_a\eta + \dot
c\partial_c\eta)\eta w,
w\ra&&\leftarrow\text{I} + \text{II} +\text{III}\\
& -\la\partial_t w,6\eta w^2+2w^3\ra-2\la (\dot a\partial_a\eta+\dot
c\partial_c\eta),w^3\ra &&\leftarrow\text{IV}+\text{V}
\end{align*}
Substitute \eqref{E:lin-w} into I:
\begin{align*}
\text{I} &=
\la\mathcal{L}w,\partial_x(\mathcal{L}w)\ra-c^2\la\mathcal{L}w,\partial_xw\ra
- \la\mathcal{L}w,\partial_x(6\eta w^2 +2w^3)\ra
+\la\mathcal{L}w,\epsilon Vw\ra -\la\mathcal{L}w,F_0\ra\\
&=\text{IA}+\text{IB}+\text{IC}+\text{ID}+\text{IE}
\end{align*}
First, $\text{IA} = 0$. Integration by parts yields $\text{IB} =
-6c^2\la\eta\eta_x,w^2\ra$. By the boundedness of $V$ and $V'$,
$$\text{ID} \lesssim \epsilon\|w\|_{H^1_x}^2\,,$$
and since $\mathcal{L}(TM)\subset TM$ (by direct computation), we
have
$$\text{IE} = -\la\mathcal{L}w, \Pi F_0\ra
-\la\mathcal{L}w,\Pi^\perp F_0\ra=-\la\mathcal{L}w,\Pi^\perp
F_0\ra\,,$$ but by \eqref{E:F0-perp}
$$|\la\mathcal{L}w,\Pi^\perp F_0\ra |\lesssim \epsilon
\|e^{-\alpha|x-a|}w\|_{H^1_x}\,,$$ hence
$$|\text{IE}|\lesssim\epsilon
\|e^{-\alpha|x-a|}w\|_{H^1_x}\,.$$ Combining, we obtain
\begin{align}
\label{E:e10} \text{I} &= \text{IB} + \text{IC} +
O\left(\epsilon\|w\|_{H^1_x}^2+\epsilon\|e^{-\alpha|x-a|}w\|_{H^1_x}\right)\\
\notag&= -6c^2\la\eta\eta_x,w^2\ra -
\la\mathcal{L}w,\partial_x(6\eta w^2 +2w^3)\ra +
O\left(\epsilon\|w\|_{H^1_x}^2+\epsilon\|e^{-\alpha|x-a|}w\|_{H^1_x}\right)
\,.
\end{align}
Substituting \eqref{E:lin-w} into IV, we have
\begin{equation}
\label{E:IC+IV}\text{IC}+\text{IV} = -\left\la\partial_x(-c^2w-6\eta
w^2-2w^3) + \epsilon Vw - F_0\,, 6\eta w^2 + 2w^3\right\ra\,.
\end{equation}
By \eqref{E:eff-dyn2}, we have
\begin{equation}
\label{E:est-dot-ac}|\dot a-c^2|\lesssim\epsilon +
\|e^{-\alpha|x-a|}w\|_{H^1_x}^2\,,\qquad |\dot c|\lesssim\epsilon +
\|e^{-\alpha|x-a|}w\|_{H^1_x}^2\,,
\end{equation}
hence
$$|\la F_0\,, 6\eta w^2 +
2w^3\ra|\lesssim\epsilon\|w\|_{H^1_x}^2 +
\|w\|_{H^1_x}^2\|e^{-\alpha|x-a|}w\|_{H^1_x}^2\,.$$ Note
$$-\la\partial_x(-c^2 w), 6\eta w^2 + 2w^3\ra = -2c^2\int\eta' w^3\,dx\,.$$
Estimating the rest of the terms in \eqref{E:IC+IV} using
Cauchy-Schwarz and that $\|w\|_{L^\infty_x}\lesssim\|w\|_{H^1_x}$,
we obtain
\begin{equation}
\label{E:e20}\text{IC}+\text{IV} = -2c^2\la\eta',w^3\ra +
O(\epsilon\|w\|_{H^1_x}^2 +
\|e^{-\alpha|x-a|}w\|_{H^1_x}^2\|w\|_{H^1_x}^2 + \|w\|_{H^1_x}^6)\,.
\end{equation}
By \eqref{E:est-dot-ac} again, and that
$\partial_x\eta=-\partial_a\eta$, we have
\begin{equation}
\label{E:e30} \text{II}+\text{V} = 2\dot a\la\eta', w^3\ra +
O(\epsilon\|w\|_{H^1_x}^2 +
\|e^{-\alpha|x-a|}w\|_{H^1_x}^2\|w\|_{H^1_x}^2)\,,
\end{equation}
and
\begin{equation}
\begin{aligned}
\label{E:e40} \text{IB}+\text{III} =& 6(\dot
a-c^2)\la\eta\eta_x,w^2\ra-6\la\dot
c(\partial_c\eta)\eta,w^2\ra\\\lesssim&\epsilon\|w\|_{H^1_x}^2 +
\|e^{-\alpha|x-a|}w\|_{H^1_x}^2\|w\|_{H^1_x}^2\,.
\end{aligned}
\end{equation}

Apply \eqref{E:est-dot-ac} again to the sum of \eqref{E:e20} and
\eqref{E:e30}, then combine with \eqref{E:e10} and \eqref{E:e40}, we
can obtain \eqref{E:energy}.

\end{proof}

\section{Proof of the main theorems}
\label{S:pf-thm}

First, we give the proof of Theorem \ref{T:main1}.

Let $[0,T']$ be the maximal time interval so that
\begin{equation}
\label{E:bs-assump}
\|w\|_{L_{[0,T]}^\infty H_x^1} \leq \mu \la t \ra^{-1/4}
\end{equation}
for $\mu>0$ chosen small enough to ensure the validity of Lemmas
\ref{L:orth-decomp}, \ref{L:ODEs}, \ref{L:app-virial}, and
\ref{L:energy}, and also small enough to beat some constants in the
estimates that follow (as explained below).

Let
$$\mathcal{V}(t) \defeq \int_0^t \|e^{-\alpha |x-a(s)|} w(s) \|_{H_x^1}^2 \, ds \,, \qquad \mathcal{F}(t)
\defeq \sup_{0\leq s \leq t} \|w(s) \|_{H_x^1}^2 \,.$$
Integrating the local virial estimate \eqref{E:app-virial-diff} gives
\begin{equation}
\label{E:loc-vir-bs}
\mathcal{V}(t) \lesssim  \mathcal{F}(t) +  \epsilon^2 t +  \epsilon\int_0^t \mathcal{F}(s) \,ds \,.
\end{equation}
Integrating \eqref{E:energy} over $0\leq t\leq \tau$ yields
$$\mathcal{E}(\tau ) \leq \mathcal{E}(0) + \epsilon\int_0^\tau \mathcal{F}(s) \, ds +
\epsilon \tau^{1/2}\mathcal{V}(\tau) + \mathcal{F}(\tau)\mathcal{V}(\tau) + \tau\mathcal{F}(\tau)^3 \,.$$
Using that $\mathcal{E}(\tau)\sim \|w(\tau)\|_{H_x^1}^2$, and then taking the sup of the above estimate over
$0\leq \tau \leq t$, we obtain
$$\mathcal{F}(t) \lesssim \mathcal{F}(0)+ \epsilon \int_0^t \mathcal{F}(s) \,ds +  \epsilon t^{1/2} \mathcal{V}(t)^{1/2} +
\mathcal{F}(t)\mathcal{V}(t) +  t \mathcal{F}(t)^3$$ By
\eqref{E:bs-assump} and the estimate $\epsilon
t^{1/2}\mathcal{V}(t)^{1/2} \leq \mu^{-1}\epsilon^2 t + \mu
\mathcal{V}(t)$ this implies
$$\mathcal{F}(t) \lesssim  \epsilon \int_0^t \mathcal{F}(s) \,ds +
\mathcal{F}(0)+ \mu^{-1}\epsilon^2 t + \mu \mathcal{V}(t)$$
Substituting \eqref{E:loc-vir-bs} into here, taking $\mu$
(introduced in \eqref{E:bs-assump} above) small enough to beat the
implicit constants,
\begin{equation}
\label{E:step100}
\mathcal{F}(t) \lesssim  \epsilon \int_0^t \mathcal{F}(s) \,ds + \mathcal{F}(0) + \epsilon^2 t \,.
\end{equation}
Hence, for some $\kappa>0$,
$$\frac{d}{dt} \left(e^{-\kappa \epsilon t} \int_0^t \mathcal{F}(s) \,ds \right) \leq e^{-\kappa \epsilon t}
(\mathcal{F}(0)+\epsilon^2 t)$$
Integrating yields
$$ \int_0^t \mathcal{F}(s) \,ds \lesssim (e^{\kappa \epsilon t}-1)(\epsilon^{-1}\mathcal{F}(0) + 1)$$
Substituting this back into \eqref{E:step100},
$$\mathcal{F}(t) \lesssim e^{\kappa \epsilon t}\mathcal{F}(0)
+ \epsilon( (e^{\kappa \epsilon t}-1) + \epsilon t)$$
For the second term, we might as well bound $(e^{\kappa \epsilon t}-1) +
\epsilon t \lesssim \epsilon t e^{\kappa \epsilon t}$, so
$$\mathcal{F}(t) \lesssim e^{\kappa \epsilon t}(\mathcal{F}(0)+ \epsilon^2 t)$$
This enables us to reach time $\sigma \epsilon^{-1}\log
\epsilon^{-1}$, for $\sigma>0$ small, while still reinforcing the
bootstrap assumption \eqref{E:bs-assump}.  Returning to
\eqref{E:loc-vir-bs}, we obtain the bound for $\mathcal{V}(t)$, thus
completing the proof of \eqref{E:thm1-est-w}.  The $L^1_{[0,T]}$
estimates \eqref{E:thm1-est-par} follow from integrating
\eqref{E:eff-dyn2} in time and applying \eqref{E:thm1-est-w}. The
$L^\infty_{[0,T]}$ estimates also follow from \eqref{E:eff-dyn2} by
dropping the spatial localization in the terms on the right-hand
side of \eqref{E:eff-dyn2} and applying the bound on
$\|w\|_{L_{[0,T]}^\infty H_x^1}$ given by\eqref{E:thm1-est-w}.

Now we discuss the proof of Theorem \ref{T:main}.

Let $\tilde a$, $\tilde c$ solve the ODE system
\begin{equation*}
\left\{
\begin{aligned}
&\dot {\tilde a} - \tilde c^2 - \epsilon \tilde c^{-1}\la V\tilde\eta,
(x-\tilde a)\tilde\eta\ra = 0 \\
&\dot {\tilde c} - \epsilon\la V\tilde\eta, \tilde\eta\ra = 0
\end{aligned}
\right.
\end{equation*}
with initial data $\tilde a(0)=a_0$, $\tilde c(0) = c_0$, where
$\tilde \eta=\tilde c Q(\tilde c(x-\tilde a))$. Since $|\dot c|,\
|\dot{\tilde c}|\lesssim\epsilon$, we can assume $\delta_0<c,\
\tilde c<\delta_0^{-1}$ on $[0,T]$. Define
$$\bar a = a - \tilde a\,,\ \ \bar c = c - \tilde c\,,$$
we have
$$\la V\eta, (x-a)\eta\ra - \la V\tilde\eta,
(x-\tilde a)\tilde\eta\ra = \beta_1(a-\tilde a) + \beta_2(c-\tilde
c)\,,$$ where we have defined
\begin{equation*}
\begin{aligned}
&\beta_1 = \frac{1}{a-\tilde a}\int\left(V(\frac{x}{\tilde c}+ a) -
V(\frac{x}{\tilde c} + \tilde a)\right)
x\eta^2\,dx\,,\\
&\beta_2 = \frac{1}{c-\tilde c}\int\left(V(\frac{x}{c}+ a) -
V(\frac{x}{\tilde c} + a)\right)x\eta^2\,dx\,,
\end{aligned}
\end{equation*}
similarly,
$$\frac{1}{c}\la V\eta,\eta\ra - \frac{1}{\tilde c}\la V\tilde\eta,
\tilde\eta\ra = \gamma_1(a-\tilde a) + \gamma_2(c-\tilde c)\,,$$
where
\begin{equation*}
\begin{aligned}
&\gamma_1 = \frac{1}{a-\tilde a}\int\left(V(\frac{x}{\tilde c}+ a) -
V(\frac{x}{\tilde c} + \tilde a)\right)
\eta^2\,dx\,,\\
&\gamma_2 = \frac{1}{c-\tilde c}\int\left(V(\frac{x}{c}+ a) -
V(\frac{x}{\tilde c} + a)\right)\eta^2\,dx\,,
\end{aligned}
\end{equation*}
Denote $\mathcal{R}_1$, $\mathcal{R}_2$ for the error terms in Lemma
\ref{L:ODEs}, i.e.
\begin{equation*}
\left\{
\begin{aligned}
&\dot {  a} -   c^2 - \epsilon c^{-1}\la V \eta,
(x-  a) \eta\ra -\mathcal{R}_1 = 0 \\
&\dot {  c} - \epsilon\la V \eta,  \eta\ra - \mathcal{R}_2 = 0\,,
\end{aligned}
\right.
\end{equation*}
Apply \eqref{E:thm1-est-par} to \eqref{E:eff-dyn2}, we obtain
\begin{equation}
\label{E:est-R12} \|\mathcal{R}_j\|_{L_{[0,t]}^1} \leq C(\omega +
\epsilon t^{1/2})^2 e^{C\epsilon^{1/2}t}  \,,\ \ j=1\,,2\,.
\end{equation}
Note
$$\frac{\dot c}{c}-\frac{\dot{\tilde c}}{\tilde c}=\frac{\dot{\bar
c}}{c}-\frac{\dot{\tilde c}}{c\tilde c}\bar c\,,$$ and
$$c\dot a - \tilde c\dot{\tilde a} = c\dot{\bar a} + (c-\tilde
c)\dot{\tilde a}\,,$$ denoting
$$\theta_1 = \frac{1}{c}\left[(c^2+{\tilde c}^2+c\tilde c) - ({\tilde c}^2+\epsilon{\tilde
c}^{-1} \la V\tilde\eta,(x-\tilde a)\tilde\eta\ra) +
\epsilon\beta_2\right]\,,$$ and
$$\theta_2=\frac{1}{\epsilon}\frac{\dot{\tilde c}}{\tilde c}
=\frac{1}{\tilde c}\la V\tilde\eta,\tilde\eta\ra\,,$$ we can obtain
the equation for $(\bar a,\bar c)$,
\begin{equation}
\label{E:eqn-bar-ac}
\begin{bmatrix}
\bar a\\
\bar c
\end{bmatrix}^\prime
=
\begin{bmatrix}
\epsilon\beta_1 c^{-1} & \theta_1 \\
\epsilon c\gamma_1 & \epsilon(\theta_2+c\gamma_2)
\end{bmatrix}
\begin{bmatrix}
\bar a\\
\bar c
\end{bmatrix}
+
\begin{bmatrix}
\mathcal{R}_1\\
\mathcal{R}_2
\end{bmatrix}\,.
\end{equation}
Writing $$A(t) = \begin{bmatrix}
\epsilon\beta_1 c^{-1} & \theta_1 \\
\epsilon c\gamma_1 & \epsilon(\theta_2+c\gamma_2)
\end{bmatrix}\,.
$$
From the boundedness of $\beta_j$, $\gamma_j$, $\theta_j$, $j=1,2$,
which is a result of the boundedness of $V$, $V'$, $c$ and $\tilde
c$, we have the estimate
\begin{equation}
|A(t)| \lesssim
\begin{bmatrix}
\epsilon & 1\\
\epsilon & \epsilon
\end{bmatrix}\,.
\end{equation}
Writing $p(s)=(\epsilon \bar a^2+\bar c^2)^{1/2}$, then by above
estimate
\begin{align*}|\dot p| &\lesssim\frac{1}{p}\left[\epsilon |\bar a|(\epsilon |\bar a|+|\bar c|+|\mathcal{R}_1|) + |\bar
c|(\epsilon |\bar a|+\epsilon |\bar
c|+|\mathcal{R}_2|)\right]\\&\lesssim\frac{1}{p}\left[\epsilon(\epsilon
\bar a^2+\bar c^2) + \epsilon^{1/2}(\epsilon\bar a^2+\bar
c^2)+\epsilon|\bar a| |\mathcal{R}_1|+|\bar c||\mathcal{R}_2|\right]
\\&\lesssim \epsilon^{1/2}p+\epsilon^{1/2}|\mathcal{R}_1|+|\mathcal{R}_2|\,.
\end{align*}
By Gronwall and $p(0)=0$, we obtain $$p(t)\leq
Ce^{C\epsilon^{1/2}t}\int_0^t\left(\epsilon^{1/2}|\mathcal{R}_1|+|\mathcal{R}_2|\right)(s)
\,ds\,.$$ Applying \eqref{E:est-R12}, we obtain
$$p(t)\leq Ce^{C\epsilon^{1/2}t}(\omega+\epsilon
t^{1/2})^2\,,$$ recalling the bounds on $t$ and $\omega$ in
Theorem~\ref{T:main}, this gives
$$p(t)\leq C \epsilon^{1/2}(\omega+\epsilon t^{1/2})e^{C\epsilon^{1/2}t}\,.$$
The bounds on $\bar a$ and $\bar c$ now follow from the definition
of $p$:
$$|\bar a|\leq C(\omega+\epsilon t^{1/2}) e^{C\epsilon^{1/2}t}\,,$$
$$|\bar c|\leq C \epsilon^{1/2}(\omega+\epsilon t^{1/2}) e^{C\epsilon^{1/2}t}\,.$$
Compare the above two estimates with \eqref{E:thm-est-w}, we can
conclude the proof of Theorem \ref{T:main}.

\begin{remark}
\label{R:r10} The $\epsilon^{-1/2}$ constraint on the time scale
stems from the
fact that the eigenvalues of $\begin{bmatrix} \epsilon & 1\\
\epsilon &\epsilon\end{bmatrix}$ are only of order $\epsilon^{1/2}$.
\end{remark}

\appendix
\section{Local and global well-posedness}
\label{A:wp}
% An exercise..we can replace all the following by: By exactly the same argument as in \cite{HPZ}
% Appendix A, we can prove the local and global well-posedness.

The global well-posedness for gKdV in energy space was obtained by
Kenig-Ponce-Vega in \cite{KPV}, where they introduced new and
powerful local smoothing and maximal function estimates, especially,
they proved the local well-posedness for \eqref{E:mkdv} in
$H^s(\mathbb{R})$ for $s\geq 1/4$. To prove well-posedness for
\eqref{E:pmkdv} at $H^1$ level of regularity, the full strength of
these estimates is not needed, we here follow the presentation of
\cite{HPZ} Apx. A and make necessary modifications.

Let $Q_n=[n-\tfrac12, n+\tfrac12]$, and $\tilde Q_n=[n-1,n+1]$. An
example of notation is:
$$\|u\|_{\ell_n^\infty L^2_T L^2_{Q_n}} =
\sup_n\|u\|_{L^2_{[0,T]}L^2_{Q_n}}\,.$$ Note that due to the finite
incidence of overlap, we have
$$\|u\|_{\ell_n^\infty L^2_T L^2_{Q_n}}\sim \|u\|_{\ell_n^\infty L^2_T L^2_{\tilde
Q_n}}\,.$$We omit the $\epsilon$ in \eqref{E:pmkdv}, and consider
\begin{equation}
\label{E:pmkdv2}
\partial_t u = \partial_x(-\partial_x^2 u -2u^3) +
V u\,, \qquad V\in \mathcal{C}^1_b\,.
\end{equation}
As in \cite{HPZ}, we first prove a local smoothing estimate and a
maximal function estimate (weak versions), by an integrating factor
method:
\begin{lemma}
\label{L:ls-mf} Suppose that
\begin{equation}
\label{E:linear-part} v_t + v_{xxx} - V v = f\,,
\end{equation}
then there exists $C>0$, such that if
$$T\leq C(1+\|V\|_{L^\infty_x})^{-1}\,,$$ we have the energy and
local smoothing estimates
\begin{equation}
\label{E:ls} \|v\|_{L^\infty_T L^2_x} + \|v_x\|_{\ell_n^\infty L^2_T
L^2_{Q_n}}\lesssim \|v_0\|_{L^2_x} +
\left\{\begin{aligned}&\|\partial_x^{-1}f\|_{\ell_n^1 L^2_T
L^2_{Q_n}}\\&\|f\|_{L^1_TL^2_x}\end{aligned}\right.
\end{equation}
and the maximal function estimate
\begin{equation}
\label{E:mf} \|v\|_{\ell_n^2 L^\infty_T L^2_{Q_n}} \lesssim
\|v_0\|_{L^2_x} + T^{1/2}\|v\|_{L^2_TH^1_x} +
T^{1/2}\|f\|_{L^2_TL^2_x}\,.
\end{equation}
The implicit constants are independent of $V$.
\end{lemma}
\begin{proof}
Let $\phi(x)=-\tan^{-1}(x-n)$, and set $w(x,t)=e^{\phi(x)}v(x,t)$.
By \eqref{E:linear-part},
$$\partial_t w + w_{xxx} - 3\phi' w_{xx}
+3(-\phi''+(\phi')^2)w_x+(-\phi'''+3\phi''\phi'-(\phi')^3)w-V
w=e^{\phi}f\,,$$ integrating its product with $\tfrac12 w$ over $x$,
$$\partial_t\|w\|_{L^2_x}^2 = -6\la\phi',w_x^2\ra +
\la-\phi'''+2(\phi')^3,w^2\ra + 2\la V,w^2\ra+2\la
e^{\phi}f,w\ra\,,$$ integrating this identity over $[0,T]$, and
using $\phi'(x)=-\la x-n\ra^{-2}$, we obtain
\begin{align*}
\|w(T)&\|^2_{L^2_x} + 6\|\la x-n\ra^{-1}w_x\|_{L^2_T L^2_x}^2\\&\leq
\|w_0\|^2_{L^2_x} + C_1 T(1+\|V\|_{L^\infty_x})\|w\|^2_{L^\infty_T
L^2_x} + C_1\int_0^{T}\left|\int e^{\phi} f w\,dx\right|\,dt\,,
\end{align*}
for some constant $C_1>0$, replace $T$ by $t$, and take supremum
over $t\in [0,T]$, we obtain, for $T\leq
\tfrac12C_1^{-1}(1+\|v\|_{L^\infty_x})$, the estimate
$$\|w\|^2_{L^\infty_T L^2_x} + \|\la x-n\ra^{-1}w_x\|_{L^2_T
L^2_x}^2\lesssim \|w_0\|^2_{L^2_x} + \int_0^{T}\left|\int e^{\phi} f
w\,dx\right|\,dt\,,$$ note that $0<e^{-\pi/2}\leq e^{\phi(x)}\leq
e^{\pi/2}<\infty$, we can convert the above estimate back to an
estimate for $v$:
$$\|v\|^2_{L^\infty_T L^2_x} + \|v_x\|_{L^2_T
L^2_{Q_n}}^2\lesssim \|v_0\|^2_{L^2_x} + \int_0^{T}\left|\int
e^{2\phi} f v\,dx\right|\,dt\,.$$ Estimating as
$$\int_0^{T}\left|\int
e^{2\phi} f v\,dx\right|\,dt\lesssim \|f\|_{L^1_T
L^2_x}\|v\|_{L^\infty_T L^2_x}\,,$$ and then taking the supremum in
$n$ yields the second estimate in \eqref{E:ls}. Estimating instead
as
$$\begin{aligned}\int_0^{T}\left|\int e^{2\phi} f v\,dx\right|\,dt
&= \int_0^{T}\left|\int(\partial_x^{-1}f
\partial_x(e^{2\phi} v)\,dx\right|\,dt \\&\leq \sum_m\|\partial_x^{-1}
f\|_{L^2_T L^2_{Q_m}}\|\la\partial_x\ra v\|_{L^2_T L^2_{Q_m}}\\&\leq
\|\partial_x^{-1}f\|_{\ell_m^1 L^2_T L^2_{Q_m}}\|\la\partial_x\ra
v\|_{\ell_m^\infty L^2_T L^2_{Q_m}} \end{aligned}$$ and then taking
the supremum in $n$ yields the second estimate in \eqref{E:ls}.

For the proof of estimate \eqref{E:mf}, take $\phi(x)=1$ on
$[n-\tfrac12,n+\tfrac12]$ and $0$ outside $[n-1,n+1]$, set $w =\phi
v$, and compute similarly as the above.

\end{proof}
Using estimates in the above lemma, we can prove:
\begin{theorem}[local well-posedness in $H^1_x$]
\label{T:lwp} Suppose that
\begin{equation}
\label{D:M} M\defeq \|V\|_{L^\infty_x}+\|V'\|_{L^\infty_x}
<\infty\,.
\end{equation}
For any $R\geq 1$, take
$$T\lesssim\min(M^{-1}, R^{-2})\,,$$ we have
\begin{enumerate}
    \item If $\|u_0\|_{H^1}\leq R$, there exists a solution $u(t)\in
    \mathcal{C}([0,T];H^1_x)$ to \eqref{E:pmkdv2} on $[0,T]$ with initial data
    $u_0(x)$ satisfying
    $$\|u\|_{L^\infty_T H^1_x} +\|u_{xx}\|_{\ell^\infty_n L^2_T
    L^2_{Q_n}}\lesssim R\,.$$
    \item This solution $u(t)$ is unique among all solutions in
    $\mathcal{C}([0,T];H^1_x)$.
    \item The data-to-solution map $u_0 \mapsto u(t)$ is continuous
    as a mapping $H^1 \rightarrow \mathcal{C}([0,T];H^1_x)$.
\end{enumerate}
\end{theorem}
\begin{proof}
We prove the existence by contraction in the space $X$, where
$$X = \{\,u\,|\, \|u\|_{\mathcal{C}([0,T];H^1_x)} + \|u_{xx}\|_{\ell^\infty_n L^2_T
    L^2_{Q_n}} + \|u\|_{\ell^2_n L^\infty_T
    L^2_{Q_n}}\leq CR\,\}\,,$$
where the constant $C$ is chosen large enough to ($10$ times, say)
exceed the implicit constants in Lemma~\ref{L:ls-mf}. Given $u\in
X$, let $\varphi(u)$ denote the solution to
\begin{equation}
\label{E:psi-u}
\partial_t\varphi(u) + \partial_x^3\varphi(u) - V\varphi(u) =
-2\partial_x(u^3)\,.
\end{equation}
with initial condition $\varphi(u)(0) = u_0$. A fixed point
$\varphi(u)=u$ in $X$ will solve \eqref{E:pmkdv2}.

The local smoothing estimate \eqref{E:ls} applied to $v=\varphi(u)$
and the estimate
$$\|(u^3)_x\|_{L^1_T L^2_x}\lesssim T\|u\|_{L^\infty_T H^1_x}^3$$
give the estimate
\begin{equation}
\label{E:ls-bd1} \|\varphi(u)\|_{L^\infty_T L^2_x}\lesssim
\|u_0\|_{H^1_x} + T\|u\|_{L^\infty_T H^1_x}^3\,,
\end{equation}
The maximal function estimate \eqref{E:mf} applied to $v=\varphi(u)$
and the estimate
$$\|(u^3)_x\|_{L^2_T L^2_x}\lesssim T^{1/2}\|u\|_{L^\infty_T H^1_x}^3$$
imply the estimate
\begin{equation}
\label{E:mf-bd} \|\varphi(u)\|_{\ell_n^2 L^\infty_T L^2_{Q_n}}
\lesssim \|u_0\|_{L^2_x} + T\|\varphi(u)\|_{L^\infty_T H^1_x} +
T\|u\|_{L^\infty_T H^1_x}^3\,.
\end{equation}

Now applying $\partial_x$ to \eqref{E:psi-u}, and denoting
$v=\varphi(u)_x$ instead:
$$v_t + v_{xxx} - V v = -2(u^3)_{xx} + V'\varphi(u)\,.$$
By \eqref{E:ls} again,
\begin{equation}
\label{E:psi-u-ls} \|\varphi(u)_x\|_{L^\infty_T L^2_x} +
\|\varphi(u)_{xx}\|_{\ell_n^\infty L^2_T L^2_{Q_n}}\lesssim
\|u_0\|_{H^1_x} + \|(u^3)_x\|_{\ell_n^1 L^2_T
L^2_{Q_n}}+\|V'\varphi(u)\|_{L^1_TL^2_x}\,.
\end{equation}
Applying Gagliado-Nirenberg inequality to $\phi(x)u$, where
$\phi(x)=1$ on $[n-\tfrac12,n+\tfrac12]$ and $0$ outside
$[n-1,n+1]$, we obtain (writing $Q$ for $Q_n$ and $\tilde Q$ for
$\tilde Q_n$ for the following):
$$\|u\|_{L^\infty_Q}^2\lesssim (\|u\|_{L^2_{\tilde Q}} + \|u_x\|_{L^2_{\tilde
Q}})\|u\|_{L^2_{\tilde Q}}\,,$$ hence
$$\|(u^3)_x\|_{L^2_Q}\lesssim
\|u_x\|_{L^2_Q}\|u\|_{L^\infty_Q}^2\lesssim
\|u_x\|_{L^2_Q}\|u\|_{L^2_{\tilde Q}}(\|u\|_{L^2_{\tilde Q}} +
\|u_x\|_{L^2_{\tilde Q}})\,.$$ Taking $L^2_T$ norm and applying the
H\"{o}lder inequality, we obtain
$$\|(u^3)_x\|_{L^2_T L^2_Q} \lesssim
\|u_x\|_{L^\infty_T L^2_Q}\|u\|_{L^\infty_T L^2_{\tilde
Q}}(\|u\|_{L^2_TL^2_{\tilde Q}} + \|u_x\|_{L^2_TL^2_{\tilde
Q}})\,.$$ Taking $\ell^1_n$ norm and applying the H\"{o}lder
inequality again yields
$$\|(u^3)_x\|_{\ell^1_nL^2_T L^2_Q} \lesssim
\|u_x\|_{\ell^\infty_nL^\infty_T L^2_Q}\|u\|_{\ell^2_nL^\infty_T
L^2_{\tilde Q}}(\|u\|_{\ell^2_nL^2_TL^2_{\tilde Q}} +
\|u_x\|_{\ell^2_nL^2_TL^2_{\tilde Q}})\,.$$ Using the bounds
$\|u_x\|_{\ell^\infty_nL^\infty_T L^2_Q}\lesssim \|u_x\|_{L^\infty_T
L^2_x}$,
$$\|u\|_{\ell^2_nL^2_TL^2_{\tilde Q}}\lesssim
\|u\|_{L^2_TL^2_x}\lesssim T^{1/2}\|u\|_{L^\infty_TL^2_x}$$ and
$$\|u_x\|_{\ell^2_nL^2_TL^2_{\tilde Q}}\lesssim
\|u_x\|_{L^2_TL^2_x}\lesssim T^{1/2}\|u_x\|_{L^\infty_TL^2_x}\,,$$
we obtain
$$\|(u^3)_x\|_{\ell_n^1 L^2_T L^2_{Q_n}}\lesssim
T^{1/2}\|u\|_{L^\infty_T H^1_x}^2\|u\|_{\ell_n^2 L^\infty_T
L^2_{Q_n}}\,,$$ inserting into \eqref{E:psi-u-ls},
\begin{equation}
\label{E:ls-bd2}
\begin{aligned}
 \|\varphi(u)_x\|&_{L^\infty_T L^2_x} + \|\varphi(u)_{xx}\|_{\ell_n^\infty L^2_T
L^2_{Q_n}}\\&\lesssim \|u_0\|_{H^1_x} + T^{1/2}\|u\|_{L^\infty_T
H^1_x}^2\|u\|_{\ell_n^2 L^\infty_T L^2_{Q_n}} +
T\|V'\|_{L^\infty_x}\|\varphi(u)\|_{L^\infty_T L^2_x}\,.
\end{aligned}
\end{equation}
Summing \eqref{E:ls-bd1}, \eqref{E:mf-bd} and \eqref{E:ls-bd2}, we
obtain that $\|\varphi(u)\|_X\leq CR$ if $\|u\|_X\leq CR$ provided
$T\leq C_0\min(M^{-1}, R^{-2})$, with $C_0$ small enough. Thus
$\varphi:\ X\rightarrow X$. A similar argument establishes that
$\varphi$ is a contraction on $X$.

Now suppose $u,v\in\mathcal{C}([0,T];H^1_x)$ solve \eqref{E:pmkdv2}.
By \eqref{E:mf},
\begin{equation}
\label{E:uv-mf}
\begin{aligned}\|u\|_{\ell_n^2 L^\infty_T
L^2_{Q_n}}&\lesssim\|u_0\|_{L^2_x} + T\|u\|_{L^\infty_T H^1_x} +
T\|u\|_{L^\infty_T H^1_x}^3\,,\\
\|v\|_{\ell_n^2 L^\infty_T L^2_{Q_n}}&\lesssim\|v_0\|_{L^2_x} +
T\|v\|_{L^\infty_T H^1_x} + T\|v\|_{L^\infty_T H^1_x}^3\,,
\end{aligned}
\end{equation}
Set $w=u-v$. Then, with $g=(u^3-v^3)/(u-v)=u^2+uv+v^2$, we have
\begin{equation*}
w_t+w_{xxx}+2(gw)_x-Vw=0\,.
\end{equation*}
Apply \eqref{E:ls} to $v=w_x$, we obtain
\begin{equation}
\label{E:w-ls} \|w_x\|_{L^\infty_T
L^2_x}\lesssim\|(gw)_x\|_{\ell_n^1 L^2_T L^2_{Q_n}} + \|V'w\|_{L^1_T
L^2_x}\,.
\end{equation}
The terms of $\|(gw)_x\|_{\ell_n^1 L^2_T L^2_{Q_n}}$ can be bounded
in the following manner:
\begin{align}
\label{E:apx-100}\|u_x v w\|_{\ell_n^1 L^2_T L^2_{Q_n}} &\lesssim
\|u_x\|_{\ell_n^\infty L^\infty_T L^2_{Q_n}}\|vw\|_{\ell_n^1 L^2_T
L^\infty_{Q_n}}\\
\notag &\lesssim\|u_x\|_{\ell_n^\infty L^\infty_T
L^2_{Q_n}}(\|vw\|_{\ell_n^1 L^2_T L^1_{Q_n}}+\|(vw)_x\|_{\ell_n^1
L^2_T L^1_{Q_n}})
\end{align}
The term in the parentheses is bounded by
$$\|v\|_{\ell_n^2 L^2_T L^2_{Q_n}}\|w\|_{\ell_n^2 L^\infty_T
L^2_{Q_n}} + \|v_x\|_{\ell_n^2 L^2_T L^2_{Q_n}}\|w\|_{\ell_n^2
L^\infty_T L^2_{Q_n}} + \|v\|_{\ell_n^2 L^\infty_T
L^2_{Q_n}}\|w_x\|_{\ell_n^2 L^2_T L^2_{Q_n}}$$ which by
\eqref{E:uv-mf} and
$$\|u_x\|_{\ell_n^\infty L^\infty_T L^2_{Q_n}}\lesssim\|u\|_{L^\infty_T
H^1_x}\,,\qquad \|v\|_{\ell_n^2 L^2_T L^2_{Q_n}}\lesssim
T^{1/2}\|v\|_{L^\infty_T L^2_x}$$ implies
$$\|u_x v w\|_{\ell_n^1 L^2_T L^2_{Q_n}} \lesssim_{\|u\|_{L^\infty_T
H^1_x},\|v\|_{L^\infty_T H^1_x}}T^{1/2}(\|w\|_{\ell_n^2 L^\infty_T
L^2_{Q_n}} + \|w\|_{L^\infty_T H^1_x})\,.$$ Same bounds follow for
other terms in $\|(gw)_x\|_{\ell_n^1 L^2_T L^2_{Q_n}}$, combined
with $\|V'w\|_{L^1_T L^2_x}\lesssim
T\|V'\|_{L^\infty_x}\|w\|_{L^\infty_T H^1_x}$, this establishes the
estimate
$$\|w_x\|_{L^\infty_T L^2_x}\lesssim T^{1/2}(\|w\|_{\ell_n^2 L^\infty_T
L^2_{Q_n}} + \|w\|_{L^\infty_T H^1_x})\,,$$ where the implicit
constant depends on $\|u\|_{L^\infty_T H^1_x}$ and
$\|v\|_{L^\infty_T H^1_x}$. Same estimate follows for
$\|w\|_{L^\infty_T L^2_x}$ by applying \eqref{E:ls} to $v=w$. Hence
\begin{equation}
\label{E:apx-bdd-w1} \|w\|_{L^\infty_T H^1_x}\lesssim
T^{1/2}(\|w\|_{\ell_n^2 L^\infty_T L^2_{Q_n}} + \|w\|_{L^\infty_T
H^1_x})\,,
\end{equation}
but applying \eqref{E:mf} to $v=w$ yields
\begin{equation}
\label{E:apx-200} \|w\|_{\ell_n^2 L^\infty_T L^2_{Q_n}}\lesssim
T\|w\|_{L^\infty_T H^1_x}
\end{equation}
since e.g.
$$\|uvw_x\|_{L^2_TL^2_x}\lesssim T^{1/2}\|w\|_{L^\infty_T
H^1_x}\|u\|_{L^\infty_T H^1_x}\|v\|_{L^\infty_T H^1_x}$$ which can
be proved by the same method as in \eqref{E:apx-100}, and thus
$\|(gw)_x\|_{L^2_TL^2_x}\lesssim T^{1/2}\|w\|_{L^\infty_T H^1_x}$.
Substituting \eqref{E:apx-200} into \eqref{E:apx-bdd-w1} implies
$w\equiv0$ for $T$ sufficiently small, which then establishes the
uniqueness of solutions in $\mathcal{C}([0,T];H^1_x)$. The
continuity of the data-to-solution map can be proved by a similar
argument.
\end{proof}
We now prove the global well-posedness in $H^1$ by (almost)
conservation laws.
\begin{theorem}[global well-posedness]
Suppose $M<\infty$, where M is defined in \eqref{D:M}, for $u_0\in
H^1$, there is a unique global solution $u\in
C_{loc}([0,\infty);H^1_x)$ to \eqref{E:pmkdv2} with
$\|u\|_{L^\infty_T H^1_x}$ controlled by $\|u_0\|_{H^1}$, $T$ and
$M$.
\end{theorem}
\begin{proof}
First, note from Gagliado-Nirenberg inequality,
$\|u\|_{L^4}^4\lesssim \|u\|_{L^2}^3 \|u_x\|_{L^2}$, we have
$$\|u_x\|_{L^2}^2 - \|u\|_{L^2}^3 \|u_x\|_{L^2} \leq H_0(u) \leq
\|u_x\|_{L^2}^2\,.$$ Applying Peter-Paul inequality to the
$\|u\|_{L^2}^3 \|u_x\|_{L^2}$ term gives us
$$\|u_x\|_{L^2}^2 + \|u\|_{L^2}^6\sim H_0(u) + \|u\|_{L^2}^6\,.$$
Suppose $u$ solves \eqref{E:pmkdv2}, then
\begin{equation}
\label{E:der-H0}
\begin{aligned}&\ \ \ \left|\frac{d}{dt} H_0(u)\right| = \left|\la
H_0'(u),JH_0'(u)+Vu\ra\right| = \left|\la
H_0'(u),Vu\ra\right|\\&\lesssim M(\|u_x\|_{L^2}^2 + \|u\|_{L^2}^2 +
\|u\|_{L^4}^4)\lesssim M(\|u_x\|_{L^2}^2 + \|u\|_{L^2}^2 +
\|u\|_{L^2}^3 \|u_x\|_{L^2})\\& \lesssim M(\|u_x\|_{L^2}^2 +
\|u\|_{L^2}^2 + \|u\|_{L^2}^6)\lesssim M(H_0(u)+ \|u\|_{L^2}^2 +
\|u\|_{L^2}^6)\,,
\end{aligned}
\end{equation}
on the other hand, by
\begin{equation*}
\left|\frac{d}{dt} P(u)\right| = \left|\la u , Vu\ra\right|\lesssim
MP(u)\,,
\end{equation*}
and Gronwall inequality, we obtain a bound on $\|u\|_{L^\infty_T
L^2_x}$ in terms of $\|u_0\|_{L^2}$ and $M$, combine this with
\eqref{E:der-H0}, and apply Gronwall again, we obtain a bound on
$H_0(u)$ and hence $\|u\|_{H^1_x}$.
\end{proof}
\begin{remark}
A global well-posedness in $H^k_x$ for $k\geq 1$ can in fact be
proved, provided $V\in\mathcal{C}^k_b$, by similar arguments.
\end{remark}

\end{document}